\documentclass[12pt,oneside,reqno]{amsart}

\usepackage{amssymb}
\usepackage{amsmath}
\usepackage{dsfont}
\usepackage{titlesec}
\titleformat{name=\section}{}{\thetitle.}{0.8em}{\centering\scshape}
\titlespacing*{\section}{0pt}{1.1\baselineskip}{\baselineskip}

\newtheorem{theorem}{\textbf Theorem}
\newtheorem{lemma}[theorem]{\textbf Lemma}

\theoremstyle{definition}
\newtheorem{definition}[theorem]{\textbf Definition}

\theoremstyle{remark}
\newtheorem{question}[theorem]{\textit Question}

\newcommand{\End}{\text{End}}
\newcommand{\BZ}{\mathbb{Z}}

\begin{document}

\title[Iterated Differential Polynomial Rings]{Iterated Differential Polynomial Rings over Locally Nilpotent Rings}

\author[Jin]{Steven Jin}

\address{Department of Mathematics, University of Maryland, College Park MD 20742, USA.}

\email{sjin6816@umd.edu}

\author[Shin]{Jooyoung Shin}

\address{Department of Mathematical Sciences, Kent State University, Kent OH 44242, USA.}

\email{jshin5@kent.edu}



\begin{abstract}
We study iterated differential polynomial rings over a locally nilpotent ring and show that a large class of such rings are Behrens radical. This extends results of Chebotar and Chen et al.
\end{abstract}

\maketitle

\section{Introduction}

Let $R$ be a ring. An additive map $d: R \to R$ that satisfies Leibniz's rule is called a \textit{derivation} of $R$. For a derivation $d$, the \textit{differential polynomial ring} $R[X;d]$ is given by all polynomials of form $a_nX^n+\dots +a_1X+a_0$ with $n\geq 0$ and $a_0,\dots ,a_n\in R$. Multiplication is given by $Xa=aX+d(a)$ for all $a\in R$ and extending via associativity and linearity. 
\par Recall that a ring is called \textit{Brown-McCoy radical} if it cannot be mapped onto a simple ring with identity. Similarly, a ring is called \textit{Behrens radical} if it cannot be mapped onto a ring with a non-zero idempotent.
\par In 1972, Krempa \cite{Krempa1972} showed that the K\"{o}the conjecture is equivalent to the statement that every polynomial ring over a nil ring is Jacobson radical. The problem remains open, but this equivalent formulation motivated the investigation of parallel questions for more general radical classes. For example, Puczylowski and Smoktunowicz \cite{Puczylowski1998} proved in 1998 that a polynomial ring over a nil ring is Brown-McCoy radical. This result was strengthened in 2001 by Beidar et al. \cite{Beidar2001}, who proved that a polynomial ring over a nil ring is Behrens radical. The corresponding questions for multivariate polynomial rings were open until recently. Then, in 2018, Chebotar et al. \cite{Chebotar-Ke2018} employed techniques from convex geometry to prove that a multivariate polynomial ring over a nil ring is Brown-McCoy radical. It is still unknown whether such a ring need be Behrens radical.
\par After restricting the class of base rings from nil rings to locally nilpotent rings, one can formulate analogous questions for differential polynomial rings. At a 2011 conference in Coimbra, Portugal, Shestakov asked whether a differential polynomial ring over a locally nilpotent ring is necessarily Jacobson radical. This can in some sense be viewed as the analog of the K\"{o}the conjecture for differential polynomial rings. Curiously, this statement turned out to be false; a 2014 result of Smoktunowicz and Ziembowski \cite{Smoktunowicz2014} yields a constructive counterexample. Nonetheless, pursuing a similar line of investigation as in the non-differential case, Greenfeld et al. \cite{GreenfeldToAppear} asked whether a differential polynomial ring over a locally nilpotent must be Behrens radical. This question was promptly resolved in the affirmative by Chebotar \cite{Chebotar2018} in 2018. 
\par Extending the results of Chebotar in two different directions, Chen et al. \cite{Chen2020} proved the following two theorems:

\begin{theorem} \label{theorem-1} \textnormal{[4, Theorem 1]}
Let $d_1,\dots , d_p$ be derivations of a locally nilpotent ring $R$. Let $X_1,\dots ,X_p$ be commuting variables. Then the differential polynomial ring $R[X_1,\dots ,X_p; d_1,\dots ,d_p]$ is Behrens radical. 
\end{theorem}

\begin{theorem}  \label{theorem-2}
\textnormal{[4, Theorem 2]} Let $\delta$ be a derivation of a locally nilpotent ring $R$ and let $d$ be a derivation of $R[X;\delta]$ such that:
\par (i) $d(R)\subseteq R$, 
\par (ii) $d|_R$ is locally nilpotent, and
\par (iii) $d^n(aX)-Xd^n(a)\in R$ for all $a\in R$ and positive integers $n$.
\par Then $R[X;\delta][Y;d]$ is Behrens radical. 

\end{theorem}

We remark that the proof of the latter theorem relies heavily on the assumption that $d|_R$ is locally nilpotent. 

\par We wish to expand upon this line of investigation. First we establish a definition.

\begin{definition} \label{definition-3}

Let $R$ be a ring. For all $1\leq i\leq n$, suppose that $d_i$ is a derivation of $R[X_1;d_1]\dots [X_{i-1};d_{i-1}]$. We denote $R[X_1;d_1]\dots [X_n;d_n]$ as $R[\bar{X}_n, \bar{d}_n]$. We call such a ring an \textit{iterated differential polynomial ring} over $R$.

\end{definition}

\par If $R$ is a ring without identity, let $R^*$ denote the ring given by adjoining an identity element to $R$. 

\par In this paper, we will prove the following result:

\begin{theorem} \label{theorem-4}

Suppose $R[\bar{X}_n, \bar{d}_n]$ is an iterated differential polynomial ring over a locally nilpotent ring $R$. Suppose that for all $i$ each $d_i$ can be extended to a derivation on $R^*[X_1;d_1]\dots [X_{i-1};d_{i-1}]$ such that $d_i$ restricts to a derivation on $R$ and further $d_i(X_j)\in R^*$ for all $0<j<i$. Then $R[\bar{X}_n, \bar{d}_n]$ is Behrens radical. 

\end{theorem}

We remark that one may view Theorem 4 as a unification of the results in Theorems 1 and 2. If $n=1$, we recover Chebotar's original theorem [2, Theorem 1]. If $n$ is arbitrary and the derivations $d_i$ are taken to be trivial off of $R$, we recover Theorem 1. If we set $n=2$, we retrieve a strengthened version of Theorem 2; namely, hypothesis (ii) has been removed and hypothesis (iii) has been weakened. In particular, the key ingredients used in the proof of [4, Theorem 2] are shown to be unnecessary. 

\par The results of this paper notwithstanding, there arises naturally the following question:

\begin{question}
Let $R$ be a locally nilpotent ring and $R[\bar{X}_n, \bar{d}_n]$ an iterated differential polynomial ring. Is $R[\bar{X}_n, \bar{d}_n]$ Behrens radical?
\end{question}

\section{Results}

We first set notation. For elements $a$ and $b$ of a ring $R$, we define $[a,b]_0=a$, $[a,b]_1=[a,b]=ab-ba$, and $[a,b]_k=[[a,b]_{k-1},b]$ for $k>1$. Given elements $b_1,\dots ,b_p\in R$ and non-negative integers $k_1,\dots ,k_p$, we denote by $[a,\bar{b}]_{k_1,\dots ,k_p}$ the expression $[\dots [a,b_1]_{k_1},\dots ,b_p]_{k_p}$ and denote by $\bar{b}^{k_1,\dots ,k_p}$ the expression $b_{1}^{k_1}\dots b_{p}^{k_p}$.

\par Additionally, suppose that $c_{i_1',\dots ,i_r'}\in R$ for $0\leq i_{q}'\leq i_q$ where $1\leq q\leq r$ and the $i_q$ are non-negative integers. Then, we write $$\sum_{i_1',\dots ,i_r'=0}^{i_1,\dots ,i_r}c_{i_1',\dots ,i_r'}:=\sum_{i_1'=0}^{i_1}\dots \sum_{i_r'=0}^{i_r}c_{i_1',\dots ,i_r'}.$$ Alternatively, if $i_1=\dots =i_r=s$, then we write $$\sum_{i_1',\dots ,i_r'}^s c_{i_1',\dots ,i_r'}:=\sum_{i_1',\dots ,i_r'=0}^{s,\dots ,s}c_{i_1',\dots ,i_r'}=\sum_{i_1'=0}^s \dots \sum_{i_r'=0}^s c_{i_1',\dots ,i_r'}.$$

We will now establish some preliminary lemmas. Our first lemma is an easy consequence of the Leibniz rule:

\begin{lemma}
Let $a,b,c$ be elements of a ring $R$. For any non-negative integer $k$, we have $$[ab,c]_k=\sum_{i=0}^k D_i[a,c]_i[b,c]_{k-i}$$ for some $D_i\in \BZ$. \hfill \qed
\end{lemma}

Other useful results include the following:

\begin{lemma}
For elements $a$ and $b$ in a ring $R$ and non-negative integers $r$ and $s$, we have 

\begin{align*}
    [a^r,b]_s=\sum_{w_1,\dots ,w_r = 0}^s E_{w_1,\dots ,w_r} [a,b]_{w_1}\dots [a,b]_{w_r}
\end{align*}
for some $E_{w_1,\dots ,w_r}\in \BZ$.
\end{lemma}

\begin{proof}
The cases $r=0$ and $r=1$ are trivial. The first nontrivial case is Lemma 6. We induct on $r$. By applying Lemma 6, we can see that
\begin{align*}
[a^{r+1},b]_s=\sum_{i=0}^s D_i [a^r,b]_i [a,b]_{s-i}. 
\end{align*}
for some $D_i\in \mathbb{Z}$.
Now we may apply the inductive hypothesis:
\begin{align*}
\sum_{i=0}^s D_i [a^r,b]_i [a,b]_{s-i} &= \sum_{i=0}^s  \sum_{w_1,\dots ,w_r=0}^{i}D_i  E_{w_1,\dots ,w_r}[a,b]_{w_1}\dots [a,b]_{w_r}[a,b]_{s-i}\\
&=\sum_{w_1,\dots ,w_{r+1} = 0}^s E_{w_1,\dots ,w_{r+1}} [a,b]_{w_1}\dots [a,b]_{w_{r+1}}.
\end{align*}
for some $D_i$, $E_{w_1,\dots ,w_r}$, $E_{w_1,\dots ,w_{r+1}} \in \mathbb{Z}$.
\end{proof}

\begin{lemma}
For elements $a_{1},\dots ,a_{n}$ and $b$ in a ring $R$ and non-negative integers $i_1,\dots ,i_n$ and $s$, we have 

\begin{align*}
[\bar{a}^{i_1,\dots ,i_n},b]_s=\sum_{w_{1}^{(1)},\dots ,w_{i_1}^{(1)} = 0}^s \dots \sum_{w_{1}^{(n)},\dots ,w_{i_n}^{(n)}=0}^s E_{w_{1}^{(1)},\dots ,w_{i_1}^{(1)}}^{(1)}\dots E_{w_{1}^{(n)},\dots ,w_{i_n}^{(n)}}^{(n)}\\ [a_1,b]_{w_{1}^{(1)}}\dots [a_1,b]_{w_{i_1}^{(1)}}\dots [a_n,b]_{w_{1}^{(n)}}\dots [a_n,b]_{w_{i_n}^{(n)}}
\end{align*}

for $E_{w_{1}^{(j)},\dots w_{i_j}^{(j)}}^{(j)}\in \BZ$ for $1\leq j\leq n$.
\end{lemma}

\begin{proof}
Induct on $n$. The base step is Lemma 7. Applying Lemma 6, observe that 

\begin{align*}
[\bar{a}^{i_1,\dots ,i_{n+1}},b]_s =    \sum_{j=0}^s D_j [\bar{a}^{i_1,\dots, i_n} ,b]_j[a_{n+1}^{i_{n+1}},b]_{s-j}.
\end{align*}
Then, applying the inductive hypothesis to $[\bar{a}^{i_1,\dots ,i_n},b]_j$ and the basis step to $[a_{n+1}^{i_{n+1}},b]_{s-j}$, we are done. 
\end{proof}

We will also take advantage of [4, Lemmas 4 and 5]. We recite these here for completeness.

\begin{lemma}
\textnormal{[4, Lemma 4]} Let $e,x_1,\dots ,x_p$ be elements of a ring $R$ and $n_1,\dots ,n_p$ be non-negative integers. Then $$e\bar{x}^{n_1,\dots,n_p}=\sum_{i_1=0}^{n_1}\dots \sum_{i_p=0}^{n_p}\begin{pmatrix} n_1\\ i_1\end{pmatrix}\dots \begin{pmatrix} n_p\\ i_p\end{pmatrix}\bar{x}^{i_1,\dots ,i_p}[e,\bar{x}]_{n_1-i_1,\dots ,n_p-i_p}.$$ \hfill \qed
\end{lemma}

\begin{lemma}
\textnormal{[4, Lemma 5]} Let $e,x_1,\dots ,x_p$ be elements of a ring $R$ with $e^2=e$. Then for any non-negative integers $k_1,\dots ,k_p$, we have $$[e,\bar{x}]_{k_1,\dots ,k_p}=\sum_{i_1=0}^{k_1}\dots \sum_{i_p=0}^{k_p}r_{i_1,\dots ,i_p}e[e,\bar{x}]_{i_1,\dots ,i_p}$$ for some $r_{i_1,\dots ,i_p}\in R$. \hfill \qed
\end{lemma}

An easy application of Lemmas 9 and 10 yields the following fact:

\begin{lemma}
Suppose $e,x_1,\dots ,x_n$ are elements of a ring. Then $e\bar{x}^{i_1,\dots ,i_n}$ can be written as a sum of terms each ending in $e[e,\bar{x}]_{k_1,\dots ,k_n}$ where $0\leq k_j\leq i_j$ and $1\leq j\leq n$. \hfill \qed
\end{lemma}

Finally, the following two lemmas are the technical heart of the proof of Theorem 4. 

\par Let $V$ be a $K$-vector space. Then, denote by $\End_{K}(V)$ the $K$-algebra of all linear transformation of $V$.

\begin{lemma}
Let $N$ be a subalgebra of $\End_K(V)$. Let $a, x_1,\dots ,x_n\in \End_K(V)$. Let $i_1,\dots ,i_n, k$ be non-negative integers. First, define the following sets:
\par (i) Let $A$ be the set of all $[a,x_j]_{i}$ for $1\leq j\leq n$ and $0\leq i\leq k$.
\par (ii) Suppose that we can write any $[x_1,x_j]_{w_{1}^{(1)}}\dots [x_1,x_j]_{w_{i_1}^{(1)}}\dots [x_n,x_j]_{w_{i_n}^{(n)}}$ in the form $\sum_{i_1',\dots ,i_n'=0}^{i_1,\dots ,i_n} \bar{x}^{i_1',\dots ,i_n'}b_{i_1',\dots ,i_n'}$ for any $1\leq j\leq n$ and for some $b_{i_1',\dots ,i_n'}\in End_K(V)$. 
Let $B$ be the set of the $b_{i_1',\dots ,i_n'}$ that arise in this way for $0\leq w_{s}^{(t)}\leq k$ for all $s,t$. 
\par (iii) Let $C$ be the set of all elements of form $\beta \alpha$ where $\alpha \in A$ and $\beta \in B$. 
\par Suppose that $C\subseteq N$. Then $[\bar{x}^{i_1,\dots ,i_n}a,x_{j}]_{k}$ can be written in the form $$\sum_{i'_1,\dots ,i'_n = 0}^{i_1, \dots ,i_n} \bar{x}^{i_1',\dots ,i_n'}c_{i'_1,\dots ,i'_n}$$ for some $c_{i'_1,\dots ,i'_n}\in N$ for all $1\leq j\leq n$.
\end{lemma}

\begin{proof}
By applying Lemmas 6 and 8, we obtain

\begin{align*}
[\bar{x}^{i_1,\dots ,i_n}a,x_{j}]_{k} 
& =\sum_{i=0}^k D_i [\bar{x}^{i_1,\dots ,i_n},x_j]_i [a,x_j]_{k-i}\\
& =\sum_{i=0}^k \sum_{w_{1}^{(1)},\dots ,w_{i_1}^{(1)}=0}^i \dots \sum_{w_{1}^{(n)},\dots ,w_{i_n}^{(n)}=0}^i D_iE_{w_{1}^{(1)},\dots ,w_{i_1}^{(1)}}^{(1)}\dots E_{w_{1}^{(n)},\dots , w_{i_n}^{(n)}}^{(n)}\\ &\quad \quad [x_1,x_j]_{w_{1}^{(1)}}\dots [x_1,x_j]_{w_{i_1}^{(1)}}\dots [x_n,x_j]_{w_{i_n}^{(n)}}[a,x_j]_{k-i}.
\end{align*}
for some $D_i, E_{w_{1}^{(1)},\dots ,w_{i_1}^{(1)}}^{(1)}\dots E_{w_{1}^{(n)},\dots , w_{i_n}^{(n)}}^{(n)}\in \mathbb{Z}$.

A single term of this sum is of form

\begin{align*}
&D_iE_{w_{1}^{(1)},\dots ,w_{i_1}^{(1)}}^{(1)}\dots E_{w_{1}^{(n)},\dots , w_{i_n}^{(n)}}^{(n)} [x_1,x_j]_{w_{1}^{(1)}}\dots [x_1,x_j]_{w_{i_1}^{(1)}}\dots [x_n,x_j]_{w_{i_n}^{(n)}}[a,x_j]_{k-i}\\
&= \sum_{i_1',\dots ,i_n'=0}^{i_1,\dots ,i_n}D_iE_{w_{1}^{(1)},\dots ,w_{i_1}^{(1)}}^{(1)}\dots E_{w_{1}^{(n)},\dots , w_{i_n}^{(n)}}^{(n)} \bar{x}^{i_1',\dots ,i_n'}b_{i_1',\dots ,i_n'}[a,x_j]_{k-i}\\
& =\sum_{i_1',\dots ,i_n'=0}^{i_1,\dots ,i_n}D_iE_{w_{1}^{(1)},\dots ,w_{i_1}^{(1)}}^{(1)}\dots E_{w_{1}^{(n)},\dots , w_{i_n}^{(n)}}^{(n)} \bar{x}^{i_1',\dots ,i_n'}c_{i_1',\dots ,i_n'}.
\end{align*}
for some $b_{i_1',\dots ,i_n'}\in B$ and some $c_{i_1',\dots ,i_n'}\in C\subseteq N$.
Since for any $c\in N$, we have that $zc\in N$ for all $z\in \BZ$, this concludes. 
\end{proof}

\begin{lemma}
Let $N$ be a locally nilpotent subalgebra of $\End_K(V)$. Let $a_{i_1,\dots ,i_n}, x_1,\dots ,x_n\in \End_K(V)$. Suppose $e=\sum_{i_1,\dots ,i_n=0}^{m_1,\dots ,m_n}\bar{x}^{i_1,\dots ,i_n}a_{i_1,\dots ,i_n}$ is an idempotent. Define the following sets:
\par (i) Consider the set of all $[a_{i_1,\dots ,i_n},x_j]_i$ for all $1\leq j\leq n$, $0\leq i\leq \max_s\{m_s\}$, and $0\leq i_r\leq m_r$. Call this set $A_1$. For any $a_{i_1,\dots ,i_n}$, suppose that we may write $[\bar{x}^{i_1,\dots ,i_n}a_{i_1,\dots ,i_n},x_1]_{k_1}$ as $\sum_{i'_1,\dots ,i'_n = 0}^{i_1, \dots ,i_n} \bar{x}^{i_1',\dots ,i_n'}c_{i'_1,\dots ,i'_n}$ for some $c_{i'_1,\dots ,i'_n}\in N$ for all $0\leq k_1\leq m_1$. Let the set of all $[c_{i_1',\dots i_n'},x_j]_i$ for all $1\leq j\leq n$, $0\leq i\leq \max_s\{m_s\}$, and $0\leq i_r'\leq m_r$ be called $A_2$. In this way, inductively define $A_1,\dots ,A_n$. Let $A=\bigcup_{i=0}^n A_i$. 
\par (ii) Suppose that any
$[x_1,x_j]_{w_1^{(1)}}\cdots[x_1,x_j]_{w_{i_1}^{(n)}}\cdots[x_n,x_j]_{w_{i_n}^{(n)}}$ can be written in the form $\sum_{i_1',\dots ,i_n'=0}^{i_1,\dots ,i_n} \bar{x}^{i_1',\dots ,i_n'}b_{i_1',\dots ,i_n'}$ 
for $1\leq j\leq n$ and for some $b_{i_1',\dots ,i_n'}\in End_K(V)$.
Let $B$ be the set of the $b_{i_1',\dots ,i_n'}$ that arise in this way for $0\leq w_{s}^{(t)}\leq \max_j\{m_j\}$ for all $s,t$. 
\par (iii) Let $C$ be the set of all elements of form $\beta\alpha$ where $\alpha\in A$ and $\beta\in B$. 
\par Suppose $C\subseteq N$. Then $e=0$. 

\end{lemma}

\begin{proof}
First, we remark that by Lemma 12, our assumption (i) is a valid hypothesis. Let $S$ be the subalgebra of $N$ generated by $C$. Then $S$ is nilpotent, so there exists subspaces $0=V_0\subseteq V_1\subseteq \dots \subseteq V_h=V$ such that $S(V_i)=V_{i-1}$. We claim that for any $0\leq l\leq h$ we have $e[e,\bar{x}]_{k_1,\dots ,k_n}(V_l)=0$ for all $0\leq k_j\leq m_j$ and $1\leq j\leq n$.
\par We induct on $l$. When $l=0$, the statement is clear. Before proceeding with the induction, we make the following intermediary assertion:\\

\textbf{Claim.} The element $[e,\bar{x}]_{k_1,\dots ,k_n}$ can be written in the form $$\sum_{i_1,\dots ,i_n = 0}^{m_1, \dots ,m_n} \bar{x}^{i_1,\dots ,i_n}c_{i_1,\dots ,i_n}$$ for some $c_{i_1,\dots ,i_n}\in S$.

\begin{proof} For this claim, we perform a nested induction on $n$. When $n=1$, we have 
\begin{align*}
[e,x_1]_{k_1}=\sum_{i_1,\dots ,i_n=0}^{m_1,\dots ,m_n}[\bar{x}^{i_1,\dots ,i_n}a_{i_1,\dots ,i_n}, x_1]_{k_1}.
\end{align*}
By condition (i) and the fact that $C$ generates $S$, this concludes the basis. For the inductive step, observe that 

\begin{align*}
[e,\bar{x}]_{k_1,\dots ,k_n}=\sum_{i_1,\dots ,i_n=0}^{m_1,\dots ,m_n} [[\bar{x}^{i_1,\dots ,i_n}a_{i_1,\dots ,i_n},\bar{x}]_{k_1,\dots ,k_{n-1}},x_n]_{k_n}.
\end{align*}
Applying the inductive hypothesis, this is 

\begin{align*}
\sum_{i_1,\dots ,i_n=0}^{m_1,\dots ,m_n} \sum_{i_1',\dots ,i_n'=0}^{i_1,\dots ,i_n}[\bar{x}^{i_1',\dots ,i_n'}c_{i'_1,\dots ,i'_n},x_n]_{k_n}.
\end{align*}
for some $c_{i'_1,\dots ,i'_n}\in S$.
Applying condition (i) to $[\bar{x}^{i_1',\dots ,i_n'}c_{i'_1,\dots ,i'_n},x_n]_{k_n}$, this proves our intermediary claim. 
\end{proof}

Now we proceed with the outer induction. Let $v\in V_l$. Then \begin{align*}
    e[e,\bar{x}]_{k_1,\dots ,k_n}(v)
    &=\sum_{i_1,\dots ,i_n = 0}^{m_1,\dots ,m_n} e\bar{x}^{i_1,\dots ,i_n}c_{i_1,\dots ,i_n}(v)\\
    &=\sum_{i_1,\dots ,i_n = 0}^{m_1,\dots ,m_n} e\bar{x}^{i_1,\dots ,i_n}(u_{i_1,\dots, i_n})
\end{align*}
for $u_{i_1,\dots, i_n}\in V_{l-1}$. By Lemma 11 and the inductive hypothesis, we are done. 
\end{proof}

\noindent \textit{Proof of Theorem 4.} We follow the approach of [2, 4]. Suppose $R[\bar{X}_n,\bar{d}_n]$ as in the theorem is Behrens radical. Then there exists a surjective homomorphism $\varphi$ from $R[\bar{X}_n,\bar{d}_n]$ onto a subdirectly irreducible ring $A$ such that there is a nonzero idempotent in the heart of $A$. Note that $A$ must be a prime ring whose extended centroid $K$ is a field. Let $Q$ be the Martindale right ring of quotients of $A$. 
\par Let $x_i:A\to A$ be maps given by $x_i(\varphi(t)):=\varphi(X_it)$ for all $t\in R[\bar{X}_n,\bar{d}_n]$ where $1\leq  i\leq n$. We claim that the $x_i$ are well-defined. Suppose $\varphi(t)=0$ and $\varphi(X_it)\neq 0$. Since $A$ is prime, there must be $t'\in R[\bar{X}_n,\bar{d}_n]$ such that $\varphi(t')\varphi(X_it)\neq 0$. We also have 
\begin{align*}
    \varphi(t')\varphi(X_it)&=\varphi(t'X_it)\\
    &=\varphi(t'X_i)\varphi(t)\\
    &=0,
\end{align*}
which is a contradiction. Note that the $x_i$ are endomorphisms of right $A$-modules, so all $x_i$ are in $Q$. Let the subring of $Q$ generated by $A$ and the $x_i$ be denoted by $A'$. Let $R'$ be the subring of $R^*[\bar{X}_n,\bar{d}_n]$ generated by $R[\bar{X}_n,\bar{d}_n]$ and $X_{i}^j$ for all $1\leq i\leq n$ and all $0\leq j$. Let $\psi:R'\to A'$ be an additive map such that $\psi(X_{i}^j)=x_{i}^j$ and $\psi(t)=\varphi(t)$ for all $t\in R[\bar{X}_n,\bar{d}_n]$. Note that $\psi$ is a homomorphism extending $\varphi$. We can write a nonzero idempotent $e\in A\subseteq A'$ as 
\begin{align*}
    e&=\varphi\left(\sum_{i_1=0}^{m_1}\dots \sum_{i_n=0}^{m_n} X_{1}^{i_1}\dots X_{n}^{i_n} r_{i_1,\dots ,i_n} \right)\\
    &=\psi\left(\sum_{i_1=0}^{m_1}\dots \sum_{i_n=0}^{m_n} X_{1}^{i_1}\dots X_{n}^{i_n} r_{i_1,\dots ,i_n} \right)\\
    &=\sum_{i_1=0}^{m_1}\dots \sum_{i_n=0}^{m_n} \bar{x}^{i_1,\dots ,i_n} a_{i_1,\dots ,i_n}
\end{align*}
where the $m_j$ are non-negative integers, $r_{i_1,\dots , i_n}\in R$, and $\psi(r_{i_1,\dots i_n})=a_{i_1,\dots ,i_n}$. Let $D$ be the subring of $A'$ generated by all $x_i$ and all $a_{i_1, \dots ,i_n}$. Let $B=D\cap \psi(R)$. Note that $B$ and the subalgebra $BK$ of $Q$ are locally nilpotent. The subalgebra $DK$ of $A'K$ is finitely generated, so it can embedded into $\End_K(V)$ for some $K$-vector space $V$. Then we can assume that $x_i \in \End_K(V)$. Finally, we have that $N=BK$ is locally nilpotent and $e=\sum_{i_1=0}^{m_1}\dots \sum_{i_n=0}^{m_n} \bar{x}^{i_1,\dots ,i_n} a_{i_1,\dots ,i_n}$ is a nonzero idempotent. Applying Lemma 13, we have a contradiction. \hfill \qed

\section{Acknowledgements}

We would like to thank the referee for her or his useful suggestions. Also, we would like to thank Prof. Mikhail Chebotar for his careful assistance and guidance, as well as his kind-hearted encouragement and support. We also extend our gratitude to the Department of Mathematical Sciences at Kent State for virtually hosting the NSF REU under which this research was conducted. The authors are supported in part by NSF grant DMS-1653002.


\begin{thebibliography}{99}

\bibitem{Beidar2001} K. I. Beidar, Y. Fong, and E. R. Puczylowski. Polynomial rings over nil ringscannot be homomorphically mapped onto rings with nonzero idempotents, \textit{J. Algebra} \textbf{238} (2001), 389-399. 

\bibitem{Chebotar2018} M. Chebotar. On differential polynomial rings over locally nilpotent rings, \textit{Israel J. Math.} \textbf{227} (2018), 233-238. 

\bibitem{Chebotar-Ke2018} M. Chebotar, W.-F. Ke, P.-H. Lee, and E. R. Puczylowski. On polynomial rings over nil rings in several variables and the central closure of prime nil rings, \textit{Israel J. Math.} \textbf{223} (2018), 309-322. 

\bibitem{Chen2020} F. Y. Chen, H. Hagan, and A. Wang. Differential polynomial rings in several variables over locally nilpotent rings, \textit{Internat. J. Algebra Comput.} \textbf{30} (2020), 117-123..

\bibitem{GreenfeldToAppear} B. Greenfeld, A. Smoktunowicz, and M. Ziembowski. Five solved problems on radicals of Ore extensions, \textit{Publ. Mat.} \textbf{63} (2019), 423-444.

\bibitem{Krempa1972} J. Krempa.  Logical connections among some open problems in non-commutative rings, \textit{Fund. Math.} \textbf{76} (1972), 275-288.

\bibitem{Puczylowski1998} E. R. Puczylowski and A. Smoktunowicz. On maximal ideals and the Brown-McCoy radical of polynomial rings, \textit{Comm. Algebra} \textbf{26} (1998), 2473-2482.

\bibitem{Smoktunowicz2014} A. Smoktunowicz and M. Ziembowski. Differential polynomial rings over locally nilpotent rings need not be Jacobson radical, \textit{J. Algebra} \textbf{412} (2014), 207-217. 

\end{thebibliography}
\end{document}